\documentclass[absolute]{mymathart}
\usepackage[theorems]{mymathmacros}
\usepackage[usenames,dvipsnames]{color}
\usepackage{subfig}
\usepackage[shortlabels]{enumitem}
\usepackage{hyperref}
\usepackage{amsmath}
\usepackage{amssymb}
\usepackage{graphicx}
\usepackage{epigraph}
\usepackage{amsfonts,amssymb,color}
\usepackage[mathscr]{eucal}
\usepackage{amsmath, amsthm}
\usepackage{mathrsfs}
\usepackage{amsbsy}
\usepackage{float}
\usepackage{verbatim}
\usepackage{amsfonts} 
\usepackage{pifont}

\title[Full dimesnsional sets of reals]{Full dimesnsional sets of reals whose sums of partial quotients increase in certain speed}

\author{Liangang Ma}

\address{Dept.\ of Mathematical Sciences, Binzhou University, Huanghe 5th Road No. 391, City of Binzhou 256600, Shandong Province, P. R. China} 
\email{maliangang000@163.com}
\thanks{The author would like to give his deep thanks to Prof. Baowei Wang as well as the two anonymous referees, who kindly help to improve several important results throughout the paper and well-presentation of the work.}    

\newtheorem{theorem}[subsection]{Theorem}
\newtheorem{lemma}[subsection]{Lemma}
\newtheorem{proposition}[subsection]{Proposition}
\newtheorem{corollary}[subsection]{Corollary}
\newtheorem*{Good's corollary}{Good's Corollary}
\newtheorem*{WW's Theorem}{WW's Theorem}
\newtheorem*{CV's theorem}{CV's Theorem}
\newtheorem*{FLWW's lemma}{FLWW's Lemma}
\newtheorem*{LR's Lemma}{LR's Lemma}
\newtheorem*{IJ's proposition}{IJ's Proposition}
\newtheorem*{Comparison Lemma}{Comparison Lemma}

\numberwithin{equation}{section}

\renewcommand{\P}{\mathcal{P}}

\begin{document} 

\begin{abstract}
 For a real $x\in(0,1)\setminus\mathbb{Q}$, let $x=[a_1(x),a_2(x),\cdots]$ be its continued fraction expansion. Let $s_n(x)=\sum_{j=1}^n a_j(x)$. The Hausdorff dimensions of the level sets
 \begin{center}
 $E_{\varphi(n),\alpha}:=\{x\in(0,1): \lim_{n\rightarrow\infty}\frac{s_n(x)}{\varphi(n)}=\alpha\}$
 \end{center}
for $\alpha\geq 0$ and a non-decreasing sequence $\{\varphi(n)\}_{n=1}^\infty$ have been studied by E. Cesaratto, B. Vall\'ee, J. Wu, J. Xu, G. Iommi, T. Jordan, L. Liao, M. Rams \emph{et al}. In this work we carry out a kind of inverse project of their work, that is, we consider the conditions on $\varphi(n)$ under which one can expect a $1$-dimensional set $E_{\varphi(n),\alpha}$. We give certain upper and lower bounds on the increasing speed of $\varphi(n)$ when  $E_{\varphi(n),\alpha}$ is of Hausdorff dimension 1 and a new class of sequences $\{\varphi(n)\}_{n=1}^\infty$ such that $E_{\varphi(n),\alpha}$ is of full dimension. For an irregular sequence $\{\varphi(n)\}_{n=1}^\infty$, a full dimensional set $E_{\varphi(n),\alpha}$ is impossible.
 \end{abstract}
 
 \maketitle

\section{Introduction}
Let $I=[0,1]$ be the unit interval. For a real $x\in I\setminus \mathbb{Q}$, let
\begin{equation}
x=[a_1(x),a_2(x),a_3(x),\cdots]:=\cfrac{1}{a_1(x)+\cfrac{1}{a_2(x)+\cfrac{1}{a_3(x)+_{\ddots}}}}
\end{equation}
be its continued fraction expansion. Let
\begin{center}
$G(x)=\{  \frac{1}{x} \}=\frac{1}{x}-[\frac{1}{x}]$
\end{center}
be the Gauss map, with the two symbols $\{\ \}$ and $[\ ]$ being the fractional and integral part of the number. $G(x)$ is conjugated to a shift map on a countable alphabet. Let 
\begin{center}
$s_n(x)=\sum_{j=1}^n a_j(x)$
\end{center} 
be the sum of the first $n$ partial quotients, $n\in\mathbb{N}$. We focus on the limit behaviors of $s_n(x)$ in this work. According to A. Ya. Khinchin [Khi],
\begin{center}
$\lim_{n\rightarrow\infty}\frac{s_n(x)}{n}=\infty$
\end{center}
almost everywhere with respect to Lebesgue measure. In 1988, W. Philipp \cite[Theorem 1]{Phi} strengthened Khinchin's result by showing that, for a sequence of positive numbers $\varphi(n)$ such that $\frac{\varphi(n)}{n}$ is non-decreasing,
\begin{center}
$\lim_{n\rightarrow\infty}\frac{s_n}{\varphi(n)}=0$ or $\limsup_{n\rightarrow\infty}\frac{s_n}{\varphi(n)}=\infty$ a. e.
\end{center} 
according to whether $\sum_{n=1}^\infty\frac{1}{\varphi(n)}<\infty$ or $=\infty$. His proof relies on  the theory of mixing random vertors or triangular arrays. As to subsets of the residual set, which are all of measure 0, it turns out that the Hausdorff dimension is a useful tool to distinguish their sizes. The initial explorer of Hausdorff dimensions of related sets is A. Besicovitch \cite{Bes}. Since then there are continuously other dimensional results, from the distribution of various terms to regularity of these distributions with respect to various levels. It turns out that the dimensional results in continued fractions provide plentiful and intricate examples and tests for various theories of fractal geometry \cite{Man}.  

From the point view of dynamical systems $(I, G(x))$, E. Cesaratto and B. Vall\'ee \cite{CV},  G. Iommi and T. Jordan \cite{IJ1, IJ2} got interesting results on Hausdorff dimensions of the sets 
\begin{center}
$\{x: \limsup_{n\rightarrow\infty}\frac{s_n(x)}{n}\leq \alpha\}$ or $\{x: \lim_{n\rightarrow\infty}\frac{s_n(x)}{n}=\alpha\}$
\end{center}
for $\alpha\in [0,\infty)$, as applications of their more comprehensive results in their more general contexts (the case $\alpha=\infty$ is also computed in \cite{IJ1}).

The \emph{level-$\alpha$ sets} 
\begin{equation}
E_{\varphi(n),\alpha}:=\{x\in(0,1): \lim_{n\rightarrow\infty}\frac{s_n(x)}{\varphi(n)}=\alpha\}
\end{equation} 
for $\alpha\in [0,\infty)$ and non-decreasing $\varphi(n), n\in\mathbb{N}$ are considered by J. Xu \cite{Xu}, J. Wu and J. Xu \cite{WX1}, as well as L. Liao and M. Rams \cite{LR1}. In the following we will often use the abbreviation 
\begin{center}
$E_{\varphi(n)}:=E_{\varphi(n),1}$ 
\end{center}
for the level-$1$ set of a sequence  $\{\varphi(n)\}_{n=1}^\infty$.  For a set $E\subset [0,1]$, let $dim_H E$ be its Hausdorff dimension. They proved various results on $dim_H E_{\varphi(n),\alpha}$ with different increasing speed $\varphi(n)$. For example, in \cite{WX1},  Wu and Xu showed that
\begin{center}
$dim_H E_{n\log n,\alpha}=1$
\end{center}  
for any $\alpha\geq 0$. They also gave more sequences $\varphi(n)$ such that $dim_H E_{\varphi(n),\alpha}=1$ in \cite[4]{WX1} with the restriction that
\begin{equation}
\limsup_{n\rightarrow\infty}\frac{\log\log\varphi(n)}{\log n}<\frac{1}{2}.
\end{equation} 
In the case $\limsup_{n\rightarrow\infty}\frac{\log\log\varphi(n)}{\log n}=\frac{1}{2}$, Liao and Rams \cite[Theorem 1.2]{LR1} proved that a full dimensional set is still possible for some sequences $\varphi(n), n\in\mathbb{N}$. In this paper we continue the search for non-decreasing sequences $\varphi(n)$ such that $dim_H E_{\varphi(n),\alpha}=1$, as an attempt to exhausting sequences with this property. We first show that
\begin{theorem}\label{Theorem1}
For a non-decreasing sequence $\{\varphi(n)\}_{n=1}^\infty$ and a positive real $\alpha$, if $dim_H E_{\varphi(n),\alpha}=1$, then 
\begin{center}
$\lim_{n\rightarrow\infty}\frac{\varphi(n)}{n}=\infty$ and $\liminf_{n\rightarrow\infty}\frac{\log\varphi(n)}{n}=0$.
\end{center}
\end{theorem} 
Then we continue to point out that 
\begin{theorem}\label{Theorem2}
For every $\alpha, \beta\in [0,1)$, there exists a non-decreasing sequence $\varphi(n), n\in\mathbb{N}$, such that
\begin{center}
$\limsup_{n\rightarrow\infty}\frac{\log\log\varphi(n)}{\log n}=\beta$  and $dim_H E_{\varphi(n),\alpha}=1$.
\end{center}
\end{theorem}
\begin{remark}
Examples of $\varphi(n)$ with $\limsup_{n\rightarrow\infty}\frac{\log\log\varphi(n)}{\log n}\in [0,1/2]$ have been given in \cite{WX1} and \cite{LR1} according to their results mentioned before. Our examples (together with \cite[Theorem 1.2]{LR1}) with $\limsup_{n\rightarrow\infty}\frac{\log\log\varphi(n)}{\log n}\in (1/2,1)$ show how big effect a mild change on the growth rate of $\varphi(n)$ can have on the dimension of the level sets $dim_H E_{\varphi(n),\alpha}$. One is recommended to compare these results with \cite[4]{WX1} and \cite[Theorem 1.1, 1.2]{LR1}.
\end{remark}

We call a non-decreasing sequence $\{\varphi(n)\}_{n=1}^\infty$ \emph{irregular} if it satisfies $\liminf_{n\rightarrow\infty}\frac{\varphi(n)}{n}<\infty$. We are particularly interested in the irregular sequences with $\limsup_{n\rightarrow\infty}\frac{\varphi(n)}{n}=\infty$, which grows slightly slower than sequences with $\lim_{n\rightarrow\infty}\frac{\varphi(n)}{n}=\infty$. In Section \ref{section3}, we will show that 
\begin{theorem}\label{Theorem4}
For any small $\epsilon>0$ and a positive real $\alpha$, there exists a sequence $\{\varphi(n)\}_{n=1}^\infty$ satisfying 
\begin{center}
$\liminf_{n\rightarrow\infty}\frac{\varphi(n)}{n}<\infty$ and $\limsup_{n\rightarrow\infty}\frac{\varphi(n)}{n}=\infty$, 
\end{center}
such that
\begin{center}
$1-dim_H E_{\varphi(n),\alpha}< \epsilon$.
\end{center}
\end{theorem}

\begin{theorem}\label{Theorem5}
For an irregular sequence $\{\varphi(n)\}_{n=1}^\infty$ and a non-negative real $\alpha$, we have 
\begin{center}
$dim_H E_{\varphi(n),\alpha}<1$.
\end{center}

\end{theorem}

Theorem \ref{Theorem5} is proved by the first referee of this paper. As proofs of these results for any $\alpha\in(0,\infty)$ are of completely same processes (case $\alpha=0$ is usually trivial by some known results), we only deal with dimension of the level-$1$ set $E_{\varphi(n)}$ instead of the level-$\alpha$ set $E_{\varphi(n),\alpha}$ in the following.

\section{Some notations and established results}\label{section4}
By a \emph{rank-$n$ basic interval} we mean
\begin{center}
$I_n(a_1,a_2,\cdots, a_n):=\{x\in (0,1): a_1(x)=a_1, \cdots, a_n(x)=a_n\}$
\end{center}
for a fixed sequence of positive integers $a_1, a_2,\cdots,a_n\in\mathbb{N}$. Its length can be explicitly expressed as a function of $\{a_1,a_2,\cdots, a_n\}$. If one sets 
\begin{center}
$p_{-1}=1, p_0=0, q_{-1}=0, q_0=1$ 
\end{center}
and 
\begin{center}
$p_n=a_np_{n-1}+p_{n-2}, q_n=a_nq_{n-1}+q_{n-2}$ for $n\geq 1$,
\end{center}
then
\begin{center}
$|I_n(a_1,a_2,\cdots, a_n)|=\frac{1}{q_n(q_n+q_{n-1})}$.
\end{center}
Especially, we use the following estimation,
\begin{center}
$\frac{1}{2q_n^2}<|I_n(a_1,a_2,\cdots, a_n)|<\frac{1}{q_n^2}$.
\end{center}

According to I. G. Good \cite[P209]{Goo}, a covering set of intervals whose elements are all basic intervals is called a \emph{fundamental covering system}. We always use fundamental covering systems throughout the work. It is enough to restrict the coverings to be fundamental ones on discussing dimensions of sets in continued fractions. The following result says that adding or neglecting finitely many partial quotients does not affect the dimension of the set.
\begin{corollary}[Good]\label{corollary2}
Let $\{\mathcal{E}_i\}_{i=1}^\infty$ be a sequence of non-empty sets of positive integers. Let
\begin{center}
$V=\{x=[a_1,a_2,a_3,\cdots]: a_i\in\mathcal{E}_i, i\in\mathbb{N}\}$.
\end{center}
For a finite number $n\in\mathbb{N}$, let $\{\mathcal{E}'_i\}_{i=1}^n$ be a finite sequence of non-empty sets of positive integers. Let 
\begin{center}
$U=\{x=[a_1,a_2,a_3,\cdots]: a_i\in\mathcal{E}'_i\mbox{ for } 1\leq i\leq n,\ a_i\in\mathcal{E}_i\mbox{ for } i>n\}$.
\end{center}
Then $dim_H V=dim_H U$.
\end{corollary} 


For a sequence of positive numbers $\varphi(n), n\in\mathbb{N}$, let
\begin{center}
$G_{\varphi(n)}:=\{x\in(0,1): a_n(x)\geq \varphi(n)\ i.o.\ n\}$
\end{center}
in which $i.o.$ reads \emph{infinitely often}. Good \cite{Goo} have ever gave bounds on $dim_H G_{\varphi(n)}$ for some $\varphi(n)$, T. \L{}uczak \cite[Theorem]{Luc} (see also \cite{FWLT}) showed $dim_H G_{c^{b^n}}=\frac{1}{b+1}$ for any $b, c>1$. In 2008, B. Wang and J. Wu \cite{WW} determind precise values of $dim_H G_{\varphi(n)}$ for any $\varphi(n)$, which greatly strengthens Good and \L{}uczak's results. They proved that
\begin{theorem}[Wang-Wu]\label{Theorem6}
Suppose $\liminf_{n\rightarrow\infty}\frac{\log\varphi(n)}{n}=\log B$, in the case $B=\infty$, let $\liminf_{n\rightarrow\infty}\frac{\log\log\varphi(n)}{n}=\log b$, then
\begin{center}
$ dim_H G_{\varphi(n)}=\left\{
\begin{array}{lll}
1 & \mbox{ if } B=1 \\
s_B & \mbox{ if } 1<B<\infty\\
1/(1+b) & \mbox{ if } B=\infty, 1\leq b\leq\infty\\
\end{array}
\right.$ 
\end{center} 
in which $1/2<s_B<1$ is monotone decreasing for $1\leq B<\infty$.
\end{theorem}
If we define $\P(s)=\lim_{n\to\infty}\frac{1}{n}\log \sum_{(a_1,\cdots, a_n)\in\mathbb{N}^n} (\frac{1}{B^nq_n^2})^s$ as the pressure function, then $s_B=\inf\{s: \P(s)\le 0\}$ in the theorem. One is recommended to \cite{MU2} and \cite{Wal} for more general pressure functions.

 Now we introduce some notations and results by E. Cesaratto and B. Vall\'ee \cite{CV} as well as G. Iommi and T. Jordan \cite{IJ1,IJ2}. We only state their results in some simple cases, which will be enough for our uses. Their original results are in far more general contexts. Denote by $F_\alpha=\{x\in (0,1): \frac{s_n(x)}{n}\leq \alpha \mbox{ for any } n\in\mathbb{N}\}$, $\gamma$ is the Euler constant. Then according to \cite[Theorem 2]{CV},
\begin{theorem}[Cesaratto-Vall\'ee]\label{Theorem7}
For any $\theta<2$, we have $|dim_H F_{\alpha}-1|=\frac{6}{\pi^2}e^{-1-\theta}2^{-\alpha}(1+O(\theta^{-\alpha})).$
\end{theorem}
While their result is focused essentially on the exponential convergent rate of $dim_H F_{\alpha}$ to $1$ as $\alpha\rightarrow\infty$, we only use the fact $dim_H F_{\alpha}<1$ for any $\alpha\in (0,\infty)$.

For a $T$-invariant (for most of our work we use the transformation $T=G(x)$ only) probability measure $\mu$, let $h(\mu)$ be its measure theoretical entropy and $\lambda(\mu)$ be its Lyapunov exponent. Let $\mathcal{M}_T$ be the set of all $T$-invariant probability measures \cite{IJ1, IJ2}. By applications of their variational results on H-dimensions of level sets determined by the Birkhoff averages of some continuous functions to continued fractions, Iommi and Jordan show that (\cite[Corollary 6.6, Proposition 6.7]{IJ1})
\begin{proposition}[Iommi-Jordan]\label{proposition3} 
$0\leq dim_H E_{n,\alpha}<1$ is real analytic, strictly increasing for $0\leq \alpha<\infty$. Moreover,
\begin{center}
 $\lim_{\alpha\rightarrow\infty} dim_H E_{n,\alpha}=1$.
\end{center}
 
\end{proposition} 
This means that for any $0\leq r<1$ there exists an unique $\alpha_r$ such that $dim_H E_{n,\alpha_r}=r$. The Proposition will be exploited in our Proposition \ref{proposition2}. 

Now we recover a result of A. Fan, L. Liao, B. Wang and J. Wu \cite[Lemma 3.2]{FLWW1} as following.
\begin{lemma}[Fan-Liao-Wang-Wu]\label{lemma3}
Let $\{s_n\}_{n=1}^\infty$ be a sequence of positive integers tending to infinity with $s_n\geq 3$ for all $n$. Then for any positive number $N\geq 2$,
\begin{center}
$dim_H\{x\in(0,1): s_n\leq a_n(x)\leq Ns_n \mbox{ for }any\ n\in\mathbb{N}\}=\big(2+\limsup_{n\rightarrow\infty}\frac{\log s_{n+1}}{\log s_1s_2\cdots s_n}\big)^{-1}$.
\end{center}
\end{lemma} 
The lemma is generalized to the following form by Liao and Rams \cite[Lemma 2.3]{LR2}.
\begin{lemma}[Liao-Rams]\label{lemma4}
For two sequences $\vec{s}=\{s_n\}_{n\geq 1}, \vec{t}=\{t_n\}_{n\geq 1}$ with $s_n\geq 1, t_n>1$ for any $n\in\mathbb{N}$. Let 
\begin{center}
$F(\vec{s},\vec{t}):=\{x\in(0,1): s_n\leq a_n(x)\leq s_nt_n \mbox{ for any } n\in\mathbb{N}\}$.
\end{center}
Now if $F(s,t)\neq\emptyset$, $\lim_{n\rightarrow\infty}s_n\rightarrow\infty$ and $\lim_{n\rightarrow\infty}\frac{\log(t_n-1)}{\log s_n}=0$, then
\begin{center}
$dim_H F(\vec{s},\vec{t})=\big(2+\limsup_{n\rightarrow\infty}\frac{\log s_{n+1}}{\log s_1s_2\cdots s_n}\big)^{-1}$.
\end{center}
\end{lemma}
They are very useful in dealing with sets with dimensions $\leq 1/2$. For various applications of them, see \cite{FLWW1} \cite{JR} \cite{LR1} and \cite{LR2}. 

At last we remind two properties of the Gauss map $G(x)$ (in general, of smooth hyperbolic iterated function systems). The first is called bounded distortion property: there exists $K>1$ such that for every $n$, for every $a_1,\ldots,a_n$ and for every $x\in [0,1]$ we have
$$
|f_{a_1,\ldots,a_n}'(x)| \geq \frac 1K |I_n(a_1,\ldots,a_n)|,
$$
where $f_{a_1,\ldots,a_n}$ is the inverse branch of the $n$-th iteration of the Gauss map, moving $[0,1]$ to $I_n(a_1,\ldots,a_n)$. The second is exponential contraction: there exists $\lambda <1$ such that for every $n$, for every $a_1,\ldots,a_n$,
$$
|I_n(a_1,\ldots,a_n)| < \lambda^n.
$$
For basics  about iterated function systems, one is recommended to \cite{MU1} and \cite{MU2}.

\section{Bound on the upper growth rate}

In this section we prove the second necessary condition in Theorem \ref{Theorem1} on the upper growth rate of $\varphi(n), n\in\mathbb{N}$ for the set $E_{\varphi(n)}$ to be of dimension 1. Bound on the lower growth rate for sequences $\{\varphi(n)\}_{n=1}^\infty$ with full dimensional level sets will be proved in Section \ref{section3}. 
\begin{theorem}\label{Theorem3}
For a non-decreasing sequence $\varphi(n), n\in\mathbb{N}$, if $dim_H E_{\varphi(n)}=1$, then 
\begin{center}
$\liminf_{n\rightarrow\infty}\frac{\log\varphi(n)}{n}=0$.
\end{center}

\end{theorem}
\begin{proof}
We again show this by reduction to absurdity. Suppose that there exists a non-decreasing sequence $\varphi(n), n\in\mathbb{N}$, such that $\liminf_{n\rightarrow\infty}\frac{\log\varphi(n)}{n}=\liminf_{n\rightarrow\infty}\frac{\log\varphi(n)}{n+1}> 0$ and $dim_H E_{\varphi(n)}=1$. Then we can find some $B>1$ and small $\epsilon>0$, such that for some $N\in\mathbb{N}$ and all $n>N$,
\begin{center}
$\frac{\log\varphi(n)}{n+1}>\log B+\epsilon$.
\end{center} 
So we have $\varphi(n)>B^{n+1}e^{(n+1)\epsilon}$. By choosing $N$ large enough, we can guarantee that for all $n>N$,
\begin{center}
$\varphi(n)>\frac{1}{(B-1)(1-\epsilon)}B^{n+1}$.
\end{center}

Now consider the set $E_{\varphi(n)}$ for this $\varphi(n)$. If $x\in E_{\varphi(n)}$, then for $n$ large enough, without loss of generality (by Corollary \ref{corollary2}), suppose for all $n>N$,
\begin{center}
$\frac{s_n(x)}{\varphi(n)}>1-\epsilon$.
\end{center} 
Then for all $n>N$,
\begin{center}
$s_n(x)>(1-\epsilon)\varphi(n)>\frac{1}{B-1}B^{n+1}>\sum_{k=1}^n B^k$.
\end{center}

Now split $E_{\varphi(n)}$ into two sets $E_1$ and $E_2$,
\begin{center}
$E_1=\{x\in E_{\varphi(n)}: a_n(x)>B^n\ i.o.\ n\}$,
\end{center}
\begin{center}
$E_2=\{x\in E_{\varphi(n)}: a_n(x)>B^n \mbox{ for only finitely many values of }n\}$.
\end{center}
Clearly we have
\begin{equation}\label{equation1}
E_{\varphi(n)}=E_1\cup E_2.
\end{equation}
By Theorem \ref{Theorem6}, 
\begin{equation}\label{equation4}
dim_H E_1<1.
\end{equation}
For $dim_H E_2$, note that  
\begin{equation}\label{equation2}
E_2=\cup_{k=1}^\infty\{x\in E_{\varphi(n)}: a_i(x)\leq B^i \mbox{ for all }i\geq k\}
\end{equation}
We claim that 
\begin{center}
$dim_H \{x\in E_{\varphi(n)}: a_i(x)\leq B^i \mbox{ for all }i\geq k\}\leq 1/2$
\end{center}
for any $k\geq 1$. Considering (\ref{equation2}) This will force
\begin{equation}\label{equation3}
dim_H E_2\leq 1/2
\end{equation}
We only show the claim in the case $k=2$, proofs for the general cases are similar.  When $k=2$, let 
\begin{center}
$A_j=\{x\in E_{\varphi(n)}: a_1(x)=j, \mbox{ there exists some } N^*\in\mathbb{N}, \mbox{ such that for all } i>N^*, \frac{1}{2}B^i \leq a_i(x)\leq B^i \}$.
\end{center}
We have the following decomposition:
\begin{center}
$\{x\in E_{\varphi(n)}: a_i(x)\leq B^i \mbox{ for all }i\geq 2\}=\cup_{j=1}^\infty A_j.$
\end{center}
This is because if there is an infinite sequence $\{i_l\in\mathbb{N}\}_{l=1}^\infty$, such that $a_{i_l}(x)<\frac{1}{2}B^{i_l}$, then for any fixed $j$ and $x\in A_j$, the inequality
\begin{center}
$s_n(x)=\sum_{i=1}^n a_i(x)>\sum_{i=1}^n B^i$
\end{center}
can not hold for all $n>N$, considering the two restrictions $a_i(x)\leq B^i$ for all $i\geq 2$ and $a_{i_l}(x)<\frac{1}{2}B^{i_l}$ for all $l\geq 1$. Now by Lemma \ref{lemma3} or \ref{lemma4}, 
\begin{center}
$dim_H A_j\leq 1/2$.
\end{center} 
So the claim is true, which successively justifies (\ref{equation3}). 

Finally, combining (\ref{equation1}) (\ref{equation4}) and (\ref{equation3}), we get
\begin{center}
$dim_H E_{\varphi(n)}<1$
\end{center}
under the assumptions on the sequence $\varphi(n)$. This contradicts the assumption that $dim_H E_{\varphi(n)}=1$, so the theorem follows.

\end{proof}
\begin{remark}
In an earlier version of the paper, the author showed that a non-decreasing sequence $\varphi(n)$ with $dim_H E_{\varphi(n)}=1$ must satisfy  
\begin{center}
$\liminf_{n\rightarrow\infty}\frac{\log\varphi(n)}{n}\leq \log\zeta(2)$
\end{center}
in which $\zeta(\cdot)$ is the Riemann zeta function, based on \cite[Theorem 5]{Goo}. Shortly after Prof. B. Wang told me that the bound can be decreased from $\log\zeta(2)$ to $0$ by Theorem \ref{Theorem6}. The proof now is modified according to his comments. Prof. Wang also gave an easier proof of the theorem, however, we feel the old one still contains something interesting and the idea will be retrieved briefly in Corollary \ref{corollary1}, so we decide only to modify it to some extent. Here is Prof. Wang's proof of Theorem \ref{Theorem3}: suppose there exists a non-decreasing sequence $\varphi(n)$ with $dim_H E_{\varphi(n)}=1$ and $\liminf_{n\rightarrow\infty}\frac{\log\varphi(n)}{n}>0$, then for some small $\epsilon>0$ and $B>1$, we have
\begin{center}
$\varphi(n)>\frac{n}{1-\epsilon}B^n$  
\end{center}
for $n$ large enough. If $x\in E_{\varphi(n)}$, then for $n$ large enough, $s_n(x)>(1-\epsilon)\varphi(n)>nB^n$. So there exists $1\leq k\leq n$, such that $a_k(x)> B^n\geq B^k$. This implies
\begin{center}
$E_{\varphi(n)}\subset\{x\in(0,1): a_n(x)>B^n\ i.o.\ n\}$.
\end{center} 
Then by WW's Theorem, $dim_H E_{\varphi(n)}\leq s_B<1$, which contradicts the assumptions. 
\end{remark}

It is easy to see that, from WW's Theorem, we can show 
\begin{corollary}
For a non-decreasing sequence $\varphi(n), n\in\mathbb{N}$ and $1/2<s_B< 1$, if $dim_H E_{\varphi(n)}\geq s_B$, then 
\begin{center}
$\liminf_{n\rightarrow\infty}\frac{\log\varphi(n)}{n}\leq \log B$.
\end{center}
\end{corollary}
\begin{corollary}
For a non-decreasing sequence $\varphi(n), n\in\mathbb{N}$ and $s_b=\frac{1}{b+1}\leq\frac{1}{2}, 1\leq b\leq\infty$, if $dim_H E_{\varphi(n)}\geq s_b$, then 
\begin{center}
$\liminf_{n\rightarrow\infty}\frac{\log\log\varphi(n)}{n}\leq \log b$.
\end{center}
\end{corollary}
There are more discussions on the techniques used in the proof of Theorem \ref{Theorem3} in Section \ref{section1}.

\section{Proof of Theorem \ref{Theorem2}}\label{section2}
In this section we give an explicit non-decreasing sequence $\varphi(n), n\in\mathbb{N}$ satisfying the properties in Theorem \ref{Theorem2}, based on Wu and Xu's work \cite{WX1}. For $N\in\mathbb{N}$, take a sequence $\{n_k=k^N\}_{k=1}^\infty$. For $0<\beta<1$, let
\begin{center}
$ \varphi(n)=\left\{
\begin{array}{ll}
e^{n_k^{\beta}} &\ \ n=n_k \mbox{ for some } k\in\mathbb{N} \\
e^{n_k^{\beta}}+n-n_k &\ \  n_k< n<n_{k+1} \mbox{ for some } k\in\mathbb{N}\\
\end{array}
\right.$ 
\end{center}
Without special declaration we always mean this sequence by the notation $\varphi(n)$ in this section. We will show that
\begin{proposition}\label{proposition1}
For $N$ large enough and the sequence $\varphi(n)$ defined above, we have
\begin{center}
$\lim_{n\rightarrow\infty}\frac{\log\log\varphi(n)}{\log n}=\beta$  and $dim_H E_{\varphi(n)}=1$.
\end{center} 
\end{proposition}
It is easy to check that $\limsup_{n\rightarrow\infty}\frac{\log\log\varphi(n)}{\log n}=\lim_{n\rightarrow\infty}\frac{\log\log\varphi(n)}{\log n}=\beta$. In the following we will demonstrate that $dim_H E_{\varphi(n)}=1$. These examples show that our upper bound for sequences with full dimensional level sets is, in some sense, the best one. The proof follows Wu and Xu's method \cite[3]{WX1}, we only give proofs where we feel necessary. We first give some notations following Wu and Xu. For an $M\in\mathbb{N}$, let
\begin{center}
$H_M=\{x\in(0,1): a_{n_k}(x)=[e^{n_k^{\beta}}-e^{n_{k-1}^{\beta}}] \mbox{ for }k\geq 2, 1\leq a_i(x)\leq M \mbox{ for } i\neq n_k \mbox{ for any } k\geq 1\}$.
\end{center}  
We will omit the integer notation $[\ ]$ in the following for simplicity, as results will not be affected. It is easy to show that $H_M\subset E_{\varphi(n)}$. For any $n\in\mathbb{N}$, let
\begin{center}
$A_n=\{(a_1, a_2, \cdots, a_n)\in\mathbb{N}^n: a_{n_k}=e^{n_k^{\beta}}-e^{n_{k-1}^{\beta}} \mbox{ for } 2\leq k\leq n^{1/N}, 1\leq a_i\leq M \mbox{ for } 1\leq i\neq n_k\leq n \mbox{ for any } k\geq 1\}$.
\end{center}
Obviously we have
\begin{center}
$H_M=\cap_{n\geq 1}\cup_{(a_1, \cdots, a_n)\in A_n} I_n(a_1, \cdots, a_n)$.
\end{center}
Let $\mathcal{A}=\{1,2,\cdots,M\}$, $t(n)=\#\{k: k^N\leq n\}\leq n^{1/N}$. For any $(a_1, \cdots, a_n)\in A_n$, let $\overline{(a_1, \cdots, a_n)}\in\mathcal{A}^{n-t(n)}$ be the finite sequence by deleting the terms $\{a_{n_k}\}_{k=1}^{t(n)}$ in $(a_1, \cdots, a_n)$. Let $\overline{[a_1, \cdots, a_n]}$ be the finite continued fraction of $\overline{(a_1, \cdots, a_n)}$. Now we show that
\begin{lemma}
For any $\epsilon>0$, there exists $N_0$, such that for any $n\geq N_0$ and $(a_1, \cdots, a_n)\in A_n$, we have
\begin{center}
$|I_n(a_1, \cdots, a_n)|\geq |I_{n-t(n)}\overline{(a_1, \cdots, a_n)}|^{1+\epsilon}$.
\end{center}
\end{lemma}
\begin{proof}
Assume $n_{k-1}\leq n< n_k$ for some $k\in\mathbb{N}$. First, by \cite[Lemma 2.1]{Wu}, we have
\begin{equation}\label{equation5}
|I_n(a_1, \cdots, a_n)|\geq\frac{1}{2q_n^2(a_1, \cdots, a_n)}\geq\frac{1}{2\big(q_{n-t(n)}\overline{(a_1, \cdots, a_n)}(e^{n_{k-1}^{\beta}}-e^{n_{k-2}^{\beta}}+1)^{t(n)}\big)^2}.
\end{equation}
Moreover, 
\begin{equation}\label{equation6}
2(e^{n_{k-1}^{\beta}}-e^{n_{k-2}^{\beta}}+1)^{2t(n)}\leq 2(e^{n_{k-1}^{\beta}})^{2t(n)}\leq 2e^{2n^\beta(1/n^N)}\leq_{\textcircled{1}} 2^{(n-t(n)-1)\epsilon}\leq_{\textcircled{2}} q_{n-t(n)}^{2\epsilon}\overline{(a_1, \cdots, a_n)}.
\end{equation}
The inequality \textcircled{1} holds as long as $N$ is large enough, while \textcircled{2} is due to $q_n\geq 2^{\frac{n-1}{2}}$ for any finite $n$-continued fraction. Now combining (\ref{equation5}) and (\ref{equation6}) we have
\begin{center}
$|I_n(a_1, \cdots, a_n)|\geq \frac{1}{2q_{n-t(n)}^{2(1+\epsilon)}\overline{(a_1, \cdots, a_n)}} \geq |I_{n-t(n)}\overline{(a_1, \cdots, a_n)}|^{1+\epsilon}$.
\end{center}
\end{proof}
\begin{remark}
From the proof one can see that the inequlity \textcircled{1} can hold only for $\beta<1$ no matter how small $t(n)$ is (recall $t(n)$ describes density of the terms $\{\varphi(n_k): n_k\leq n\}$ in $\{\varphi(j)\}_{j=1}^n$). This is the obstacle to generalize the method to cases $\beta\geq 1$. 
\end{remark}

Now for two points $\{x,y\}\subset H_M$ and $x<y$, there is one and only one integer $n\geq 1$, such that 
\begin{center}
$x\in I_{n+1}(a_1, \cdots, a_n, l_{n+1}), y\in I_{n+1}(a_1, \cdots, a_n, r_{n+1}), 1\leq l_{n+1}\neq r_{n+1}\leq M$
\end{center}
with
\begin{center}
$(a_1, \cdots, a_n)\in A_n, (a_1, \cdots, a_n, l_{n+1})\in A_{n+1}, (a_1, \cdots, a_n, r_{n+1})\in A_{n+1}$.
\end{center}
For the distance between $x$ and $y$, similar to \cite[Lemma 3.4]{WX1}, we can show
\begin{center}
$y-x\geq \frac{1}{(M+2)^3} I_n(a_1, \cdots, a_n)$.
\end{center} 

Now we are in a position to prove Proposition \ref{proposition1}.

Proof of Proposition \ref{proposition1}: 
\begin{proof}
This is of similar process as \cite[Proof of Theorem 1.4]{WX1}, with only a change of some constants.
\end{proof}

Theorem \ref{Theorem2} follows directly from Proposition \ref{proposition1}.

\section{The irregular case-proof of Theorem \ref{Theorem1}, \ref{Theorem4} and \ref{Theorem5}}\label{section3}
While cases of sums of partial quotients grow \emph{linearly} ($\varphi(n)=\alpha n, \alpha\in[1,\infty)$) were considered in \cite{CV} \cite{IJ1, IJ2}, Wu and Xu \cite{WX1}, Liao and Rams \cite{LR1} have dealt with the dimentional problem under the regular condition $\lim_{n\rightarrow\infty}\frac{\varphi(n)}{n}=\infty$ on the growth rate. It is termed the \emph{super-linear} case in \cite{WX1}. The restriction $\lim_{n\rightarrow\infty}\frac{\varphi(n)}{n}=\infty$ does guarantee some convenience and general results (the condition is also assumed in \cite{FLWW2} and \cite{LR2}) on studying $dim_H E_{\varphi(n)}$. However, there seems little known in the irregular case, in which we pay special attention to sequences $\{\varphi(n)\}_{n=1}^\infty$ satisfying 
\begin{equation}\label{equation12}
\liminf_{n\rightarrow\infty}\frac{\varphi(n)}{n}<\infty
\end{equation}
and
\begin{equation}\label{equation13}
\limsup_{n\rightarrow\infty}\frac{\varphi(n)}{n}=\infty.
\end{equation}
simutaneously. In this section we present some results in this case, by establishing some links between results on $dim_H E_{\varphi(n)}$ with various growth rates $\varphi(n)$. However, we are not the first to deal with some irregular growth rate considering \cite[Corollary 3]{DV} in continued fractions. Note that our sequences are even more ``irregular" than ones in \cite{DV} as $\frac{\varphi(n)}{n}$ can not be non-decreasing for $\varphi(n)$ satisfying \ref{equation12} and \ref{equation13}.  

For two finite $n$-sequences $\vec{\sigma}_n=(\sigma_1,\sigma_2,\cdots,\sigma_n)$ and $\vec{\tau}_n=(\tau_1,\tau_2,\cdots,\tau_n)$ in $\mathbb{N}^n$ differs only at the subscripts $\Omega=\{n_1,n_2,\cdots,n_{t(n)}\}, 1\leq n_1<n_2<\cdots<n_{t(n)}\leq n$,  we mean       
\begin{center}
$\sigma_j=\tau_j  \mbox{ if }  j\notin\Omega$ and $\sigma_j\neq\tau_j  \mbox{ if }  j\in\Omega$
\end{center}
for $1\leq j\leq n$, $t(n)=\#\Omega$. Define the convergents, numerators and denominators of the finite $n$-sequences inductively as the case of infinite sequences at the beginning of Section \ref{section4}.
\begin{lemma}\label{lemma2}
Compare the denominators of the convergents of the two $n$-sequences $\vec{\sigma}_n$ and $\vec{\tau}_n$ differing only at the subscripts $\Omega$, we have 
\begin{center}
$\cfrac{\sigma_{n_1}\sigma_{n_2}\cdots\sigma_{n_{t(n)}}}{(\tau_{n_1}+1)(\tau_{n_2}+1)\cdots(\tau_{n_{t(n)}}+1)}<\cfrac{q_n(\vec{\sigma}_n)}{q_n(\vec{\tau}_n)}<\cfrac{(\sigma_{n_1}+1)(\sigma_{n_2}+1)\cdots(\sigma_{n_{t(n)}}+1)}{\tau_{n_1}\tau_{n_2}\cdots\tau_{n_{t(n)}}}$.
\end{center}

\end{lemma}
\begin{proof}
Let $q_j(\vec{\sigma}_n)=q_j(\sigma_1,\cdots,\sigma_j)$ and $q_j(\vec{\tau}_n)=q_j(\tau_1,\cdots,\tau_j)$ for $1\leq j\leq n$. We only show 
\begin{equation}\label{equation7}
\cfrac{q_n(\vec{\sigma}_n)}{q_n(\vec{\tau}_n)}<\cfrac{(\sigma_{n_1}+1)(\sigma_{n_2}+1)\cdots(\sigma_{n_{t(n)}}+1)}{\tau_{n_1}\tau_{n_2}\cdots\tau_{n_{t(n)}}},
\end{equation}
the other inequality can be shown similarly. First, we have 
\begin{center}
$q_j(\vec{\sigma}_n)=q_j(\vec{\tau}_n)$ 
\end{center}
for $-1\leq j<n_1$ as $\sigma_j=\tau_j$ for $-1\leq j<n_1$. When $j=n_1$, 
\begin{center}
$\cfrac{q_{n_1}(\vec{\sigma}_n)}{q_{n_1}(\vec{\tau}_n)}=\cfrac{\sigma_{n_1}q_{n_1-1}(\vec{\sigma}_n)+q_{n_1-2}(\vec{\sigma}_n)}{\tau_{n_1}q_{n_1-1}(\vec{\tau}_n)+q_{n_1-2}(\vec{\tau}_n)}<\cfrac{(\sigma_{n_1}+1)q_{n_1-1}(\vec{\sigma}_n)}{\tau_{n_1}q_{n_1-1}(\vec{\tau}_n)}=\cfrac{\sigma_{n_1}+1}{\tau_{n_1}}$.
\end{center}
When $j=n_1+1$,

$\begin{array}{llllll}
\ & \cfrac{q_{n_1+1}(\vec{\sigma}_n)}{q_{n_1+1}(\vec{\tau}_n)}-\cfrac{\sigma_{n_1}+1}{\tau_{n_1}}\\
=&\cfrac{\sigma_{n_1+1}q_{n_1}(\vec{\sigma}_n)+q_{n_1-1}(\vec{\sigma}_n)}{\tau_{n_1+1}q_{n_1}(\vec{\tau}_n)+q_{n_1-1}(\vec{\tau}_n)}-\cfrac{\sigma_{n_1}+1}{\tau_{n_1}}\\
=&\cfrac{\sigma_{n_1+1}\tau_{n_1}q_{n_1}(\vec{\sigma}_n)+q_{n_1-1}(\vec{\sigma}_n)\tau_{n_1}-(\sigma_{n_1}+1)(\tau_{n_1+1}q_{n_1}(\vec{\tau}_n)+q_{n_1-1}(\vec{\tau}_n))}{\tau_{n_1}(\tau_{n_1+1}q_{n_1}(\vec{\tau}_n)+q_{n_1-1}(\vec{\tau}_n))}\\
<&\cfrac{\sigma_{n_1+1}(\sigma_{n_1}+1)q_{n_1}(\vec{\tau}_n)+q_{n_1-1}(\vec{\sigma}_n)\tau_{n_1}-(\sigma_{n_1}+1)(\tau_{n_1+1}q_{n_1}(\vec{\tau}_n)+q_{n_1-1}(\vec{\tau}_n))}{\tau_{n_1}(\tau_{n_1+1}q_{n_1}(\vec{\tau}_n)+q_{n_1-1}(\vec{\tau}_n))}\\
<&\cfrac{(\sigma_{n_1}+1)(\sigma_{n_1+1}q_{n_1}(\vec{\tau}_n)-\tau_{n_1+1}q_{n_1}(\tau)-q_{n_1-1}(\vec{\tau}_n))}{\tau_{n_1}(\tau_{n_1+1}q_{n_1}(\vec{\tau}_n)+q_{n_1-1}(\vec{\tau}_n))}\\
=&\cfrac{-(\sigma_{n_1}+1)q_{n_1}(\vec{\tau}_n)}{\tau_{n_1}(\tau_{n_1+1}q_{n_1}(\vec{\tau}_n)+q_{n_1-1}(\vec{\tau}_n))}\\
<& 0.\\
\end{array}$

So 
\begin{center}
$\cfrac{q_{n_1+1}(\vec{\sigma}_n)}{q_{n_1+1}(\vec{\tau}_n)}<\cfrac{\sigma_{n_1}+1}{\tau_{n_1}}$.
\end{center}
Inductively one can show 
\begin{center}
$\cfrac{q_{j}(\vec{\sigma}_n)}{q_{j}(\vec{\tau}_n)}<\cfrac{\sigma_{n_1}+1}{\tau_{n_1}}$
\end{center}
for any $n_1\leq j<n_2$. When $j=n_2$,
\begin{center}
$\cfrac{q_{n_2}(\vec{\sigma}_n)}{q_{n_2}(\vec{\tau}_n)}=\cfrac{\sigma_{n_2}q_{n_2-1}(\vec{\sigma}_n)+q_{n_2-2}(\vec{\sigma}_n)}{\tau_{n_2}q_{n_2-1}(\vec{\tau}_n)+q_{n_2-2}(\vec{\tau}_n)}
<\cfrac{(\sigma_{n_2}+1)q_{n_2-1}(\vec{\sigma}_n)}{\tau_{n_2}q_{n_2-1}(\vec{\tau}_n)}
<\cfrac{(\sigma_{n_2}+1)(\sigma_{n_1}+1)}{\tau_{n_2}\tau_{n_1}}$.
\end{center}
By similar inductive steps as before we can show
\begin{center}
$\cfrac{q_{j}(\vec{\sigma}_n)}{q_{j}(\vec{\tau}_n)}<\cfrac{(\sigma_{n_2}+1)(\sigma_{n_1}+1)}{\tau_{n_2}\tau_{n_1}}$
\end{center}
for any $n_2\leq j<n_3$. The inequality \ref{equation7} holds by inductive processes on $n_l$ for $1\leq l\leq t(n)$. 
\end{proof}

Now we aim to compare lengths of the two rank-$n$ basic intervals $I_n(\vec{\sigma}_n)$ and $I_n(\vec{\tau}_n)$.
\begin{Comparison Lemma}
Suppose that for $n$ large enough
\begin{equation}\label{equation8}
\max\{\sigma_{n_1}+1,\sigma_{n_2}+1,\cdots,\sigma_{n_{t(n)}}+1,\tau_{n_1}+1,\tau_{n_2}+1,\cdots,\tau_{n_{t(n)}}+1\}\leq\psi(n)
\end{equation}
 for some $\psi(n)$. For any $\epsilon>0$, if $2^{(n-1)\epsilon}\geq 2\psi^{2t(n)}(n)$ for $n$ large enough, then
\begin{equation}\label{equation11}
|I_n(\vec{\tau}_n)|^{1+\epsilon}\leq |I_n(\vec{\sigma}_n)|\leq |I_n(\vec{\tau}_n)|^{1-\epsilon}
\end{equation}
for $n$ large enough.
\end{Comparison Lemma}
\begin{proof}
First, by Lemma \ref{lemma2},
\begin{center}
$|I_n(\vec{\sigma}_n)|\geq\cfrac{1}{2q_n^2(\vec{\sigma}_n)}\geq\cfrac{1}{2\big(\frac{(\sigma_{n_1}+1)\cdots(\sigma_{n_{t(n)}}+1)}{\tau_{n_1}\cdots\tau_{n_{t(n)}}}q_n(\vec{\tau}_n)\big)^2}$.
\end{center}
Now by \ref{equation8},
\begin{center}
$q_n(\vec{\tau}_n)^{2\epsilon}\geq 2^{(n-1)\epsilon}\geq 2\psi^{2t(n)}(n)\geq2\big(\frac{(\sigma_{n_1}+1)\cdots(\sigma_{n_{t(n)}}+1)}{\tau_{n_1}\cdots\tau_{n_{t(n)}}}\big)^2$,
\end{center}
so
\begin{equation}\label{equation9}
|I_n(\vec{\sigma}_n)|\geq \frac{1}{q_n^{2+2\epsilon}(\vec{\tau}_n)}\geq |I_n(\vec{\tau}_n)|^{1+\epsilon}
\end{equation}
By interchanging $\vec{\sigma}_n$ and $\vec{\tau}_n$, repeat the former process, we get $|I_n(\vec{\tau}_n)|\geq |I_n(\vec{\sigma}_n)|^{1+\epsilon}$, so
\begin{equation}\label{equation10}
|I_n(\vec{\sigma}_n)|\leq |I_n(\vec{\tau}_n)|^{1/{(1+\epsilon)}}\leq |I_n(\vec{\tau}_n)|^{1-\epsilon}.
\end{equation}
\ref{equation9} and \ref{equation10} justify \ref{equation11}.

\end{proof}
\begin{remark}
It is easy to see that, the lemma holds as long as we assume 
\begin{center}
$2^{(n-1)\epsilon}\geq 2\big(\frac{(\sigma_{n_1}+1)\cdots(\sigma_{n_{t(n)}}+1)}{\tau_{n_1}\cdots\tau_{n_{t(n)}}}\big)^2$ and $2^{(n-1)\epsilon}\geq2\big(\frac{(\tau_{n_1}+1)\cdots(\tau_{n_{t(n)}}+1)}{\sigma_{n_1}\cdots\sigma_{n_{t(n)}}}\big)^2$
\end{center}
for any $\epsilon>0$. Each of the two assumed inequality guarantees one of the two inequalities in \ref{equation11}. Moreover, the result can also be sharpened when considering sets of reals with large $q_n(\vec{\tau}_n)$ and/or $q_n(\vec{\sigma}_n)$ uniformly.   These give more general or sharper results, however, the one now is enough for our use in this work. Compare the lemma with \cite[Lemma 3.3]{WX1}.

\end{remark}

As a first application of the lemma, we show that
\begin{corollary}
For the sets 
\begin{center}
$E_M(\alpha):=\{x\in[0,1): a_{l^2}(x)=[4\alpha l\log l] \mbox{ for all } l\geq 2 \mbox{ and }1\leq a_i(x)\leq M \mbox{ for } i\neq l^2 \mbox{ for any } l\geq 1\}$,
\end{center}
$E_M:=\{x\in[0,1): 1\leq a_i(x)\leq M \mbox{ for any } i\geq 1\}$ in \cite{WX1} and  $H_M$ defined in the last section, we have
\begin{center}
$dim_H E_M(\alpha)=dim_H E_M=dim_H H_M$
\end{center}  
\end{corollary}
\begin{proof}
We only show the first half of the equality, the second half is similar. It is shown in \cite[Proof of theorem 1.4]{WX1} that $dim_H E_M(\alpha)\geq \frac{1}{1+\epsilon} dim_H E_M$ for any $\epsilon>0$. Let $n_l=l^2$ for $l\in\mathbb{N}$. Consider the subset of the set $E_M$,
\begin{center}
$E_M':=\{x\in[0,1): a_{l^2}(x)=1 \mbox{ for all } l\geq 2 \mbox{ and }1\leq a_i(x)\leq M \mbox{ for } i\neq l^2 \mbox{ for any } l\geq 1\}$.
\end{center} 
We will show that $dim_H E_M(\alpha)=dim_H E_M'$, which is enough to imply $dim_H E_M(\alpha)=dim_H E_M$. Obviously there is an $1-1$ correspondence between numbers in  $E_M(\alpha)$ and $E_M'$. We will show this is also the case for their fundamental coverings. For a $\delta$-covering of countable fundamental intervals $\{I_{m_j}(\sigma_1,\sigma_2,\cdots,\sigma_{m_j})\}_{j=1}^\infty$ of the set $E_M(\alpha)$, let $t(m_j):=\#\{l\in\mathbb{N}: l^2\leq m_j\}$. Obviously $t(m_j)\leq m_j^{1/2}$ for $m_j$ large enough.  Then the set of fundamental intervals
\begin{center}
$\{I_{m_j}(\tau_1,\tau_2,\cdots,\tau_{m_j}): \tau_i=\sigma_i \mbox{ if } i\neq l^2 \mbox{ for any }1\leq l\leq t(m_j) \mbox{ and } 1\leq i\leq m_j, \tau_i=1 \mbox{ if } i= l^2 \mbox{ for some }1\leq l\leq t(m_j)\}_{j=1}^\infty$
\end{center}
is a $\delta'$-covering of  $E_M'$. If $\delta$ is small enough, we can guarantee the infimum rank $\inf\{m_j\}_{j=1}^\infty$ is large enough and $\delta'$ is small enough. $(\sigma_1,\sigma_2,\cdots,\sigma_{m_j})$ and $(\tau_1,\tau_2,\cdots,\tau_{m_j})$ differs only at the subscripts $\Omega=\{n_l\}_{l=1}^{t(m_j)}$. Moreover,
\begin{center}
$\max\{\sigma_{n_1}+1,\cdots,\sigma_{n_{t(m_j)}}+1,\tau_{n_1}+1,\cdots,\tau_{n_{t(m_j)}}+1\}\leq m_j$
\end{center}
for any $m_j$. For any $\epsilon>0$, we have
\begin{center}
$2^{(m_j-1)\epsilon}\geq 2m_j^{2m_j^{1/2}}\geq 2m_j^{2t(m_j)}$.
\end{center}
By the Comparison Lemma, 
\begin{center}
$|I_n(\tau_1,\tau_2,\cdots,\tau_{m_j})|^{1+\epsilon}\leq |I_n(\sigma_1,\sigma_2,\cdots,\sigma_{m_j})|\leq |I_n(\tau_1,\tau_2,\cdots,\tau_{m_j})|^{1-\epsilon}$.
\end{center}
It is easy to see that every $\delta'$-covering of $E_M'$ of fundamental intervals $\{I_{m_j}(\tau_1,\tau_2,\cdots,\tau_{m_j})\}_{j=1}^\infty$ can be obtained in such a way through a $\delta$-covering of $E_M(\alpha)$. So we have
\begin{center}
$dim_H E_M(\alpha)=dim_H E_M'$.
\end{center}

\end{proof}

By Proposition \ref{proposition3} and the Comparison Lemma, we can demonstrate that
\begin{proposition}\label{proposition2} 
For any $0\leq r<1$, there exists a non-decreasing sequence $\{\varphi(n)\}_{n=1}^\infty$ with $\liminf_{n\rightarrow\infty}\frac{\varphi(n)}{n}=\alpha_r+e^{-1}$ and $\limsup_{n\rightarrow\infty}\frac{\varphi(n)}{n}=\infty$, such that
\begin{center}
$dim_H E_{\varphi(n)}\geq r$.
\end{center}
\end{proposition}
\begin{proof}
By Proposition \ref{proposition3}, $dim_H E_{n,\alpha_r}=r$ for any $0\leq r<1$. Now consider the following set
\begin{center}
$E_{n,\alpha_r}':=\{x\in(0,1): a_n(x)=a_n(y) \mbox{ if } n\neq l^l, a_{l^l}(x)=a_{l^l}(y)+l^{l+1}-(l-1)^l \mbox{ for any } l\in\mathbb{N} \mbox{ and some }y\in E_{n,\alpha_r}\}$.
\end{center}
There is an $1-1$ mapping between numbers in  $E_{n,\alpha_r}$ and $E_{n,\alpha_r}'$. Now suppose the set $E_{n,\alpha_r}$ has a $\delta$-covering of countable fundamental intervals $\{I_{m_j}(\sigma_1,\sigma_2,\cdots,\sigma_{m_j})\}_{j=1}^\infty$. Let $t(m_j):=\#\{l\in\mathbb{N}: l^l\leq m_j\}$, obviously $t(m_j)\leq m_j^{1/2}$ for $m_j$ large enough.  Then the sets of fundamental intervals
\begin{center}
$\{I_{m_j}(\tau_1,\tau_2,\cdots,\tau_{m_j}): \tau_i=\sigma_i \mbox{ if } i\neq l^l \mbox{ for any }1\leq l\leq t(m_j) \mbox{ and } 1\leq i\leq m_j, \tau_i=\sigma_i+l^{l+1}-(l-1)^l \mbox{ if } i= l^l \mbox{ for some }1\leq l\leq t(m_j)\}_{j=1}^\infty$
\end{center}
is a $\delta'$-covering of  $E_{n,\alpha_r}'$. If $\delta$ is small enough, we can guarantee the infimum rank $\inf\{m_j\}_{j=1}^\infty$ is large enough and $\delta'$ is small enough. As $\{I_{m_j}(\sigma_1,\sigma_2,\cdots,\sigma_{m_j})\}_{j=1}^\infty$ covers $E_{n,\alpha_r}$, we have $\sum_{i=1}^{m_j}\sigma_i\sim m_j\alpha_r$ (means $1-\epsilon<\frac{\sum_{i=1}^{m_j}\sigma_i}{m_j\alpha_r}<1+\epsilon$). Now compare $|I_{m_j}(\sigma_1,\sigma_2,\cdots,\sigma_{m_j})|$ and $|I_{m_j}(\tau_1,\tau_2,\cdots,\tau_{m_j})|$. In notations of the Comparsion Lemma, let $n_l=l^l$ for any $1\leq l\leq t(m_j)$. Note that $(\sigma_1,\sigma_2,\cdots,\sigma_{m_j})$ and $(\tau_1,\tau_2,\cdots,\tau_{m_j})$ differs only at the subscripts $\Omega=\{n_l\}_{l=1}^{t(m_j)}$. Moreover,
\begin{center}
$\max\{\sigma_{n_1}+1,\cdots,\sigma_{n_{t(m_j)}}+1,\tau_{n_1}+1,\cdots,\tau_{n_{t(m_j)}}+1\}\leq 2m_j\alpha_r+m_j^2< m_j^3$
\end{center}
for $m_j$ large enough. For any $\epsilon>0$ and $m_j$ large enough, we have
\begin{center}
$2^{(m_j-1)\epsilon}\geq 2m_j^{6m_j^{1/2}}\geq 2m_j^{6t(m_j)}$.
\end{center}
By the Comparison Lemma, 
\begin{center}
$|I_n(\tau_1,\tau_2,\cdots,\tau_{m_j})|^{1+\epsilon}\leq |I_n(\sigma_1,\sigma_2,\cdots,\sigma_{m_j})|\leq |I_n(\tau_1,\tau_2,\cdots,\tau_{m_j})|^{1-\epsilon}$.
\end{center}
It is easy to see that every $\delta'$-covering of $E_{n,\alpha_r}'$ of fundamental intervals 
\begin{center}
$\{I_{m_j}(\tau_1,\tau_2,\cdots,\tau_{m_j})\}_{j=1}^\infty$
\end{center}
can be obtained in such a way through a $\delta$-covering of $E_{n,\alpha_r}$. So we have
\begin{center}
$dim_H E_{n,\alpha_r}=dim_H E_{n,\alpha_r}'=r$.
\end{center}
Now let 
\begin{center}
$\varphi_r(n)=n\alpha_r+(l-1)^l$ for $(l-1)^{l-1}\leq n< l^l$, $l\geq 2$. 
\end{center}
It is easy to check that 
\begin{center}
$\liminf_{n\rightarrow\infty}\frac{\varphi_r(n)}{n}=\alpha_r+e^{-1}$, $\limsup_{n\rightarrow\infty}\frac{\varphi_r(n)}{n}=\infty$
\end{center}
and $E_{n,\alpha_r}'\subset E_{\varphi_r(n)}$, so $dim_H E_{\varphi_r(n)}\geq r$, which justifies the proposition.

\end{proof}
It would be an interesting question to ask whether $dim_H E_{\varphi_r(n)}= r$ or  $dim_H E_{\varphi_r(n)}> r$ for the sequence $\varphi_r(n)$ defined above. Theorem \ref{Theorem4} follows directly from Proposition \ref{proposition2}. 

Now we deal with bound on the lower growth rate for sequences $\{\varphi(n)\}_{n=1}^\infty$ with full dimensional level sets $E_{\varphi(n),\alpha}$. We first show a pioneer bound on the lower growth rate.
\begin{lemma}\label{lemma1}
For a non-decreasing sequence $\varphi(n), n\in\mathbb{N}$, if $dim_H E_{\varphi(n)}=1$, then 
\begin{center}
$\limsup_{n\rightarrow\infty}\frac{\varphi(n)}{n}=\infty$.
\end{center}
\end{lemma}
\begin{proof}
We show this by reduction to absurdity. Suppose that there exists a non-decreasing sequence $\varphi(n), n\in\mathbb{N}$, such that $\limsup_{n\rightarrow\infty}\frac{\varphi(n)}{n}=C<\infty$ and $dim_H E_{\varphi(n)}=1$. Then for this $\varphi(n)$, 
\begin{center}
$E_{\varphi(n)}\subset F_{C}$
\end{center}  
according to our notations before. By Theorem \ref{Theorem7}, $dim_H F_{C}<1$, so $dim_H E_{\varphi(n)}<1$. This contradicts the assumption $dim_H E_{\varphi(n)}=1$, so the conclusion follows.
\end{proof}

It is then natural to ask whether it is possible for the level sets of an irregular sequence $\{\varphi(n)\}_{n=1}^\infty$ to be of full dimension. This is denied by the first referee of the paper, together with the following smart proof. Denote by

\[
F(\alpha) = \{x\in [0,1]; \limsup \frac 1n \sum_{\ell=1}^n a_\ell(x) \leq \alpha\} 
\] and
\[
W(\alpha) = \{x\in [0,1]; \liminf \frac 1n \sum_{\ell=1}^n a_\ell(x) = \alpha\}
\] 
for any $\alpha\geq 0$.

\begin{theorem}\label{Theorem8}
For dimensions of the two sets $F(\alpha)$ and $G(\alpha)$, we have
\[
\dim_H W(\alpha) \leq \dim_H F(\alpha).
\]
\end{theorem}

\begin{proof}


Fix $\varepsilon>0$ and $n\in\mathbb{N}$. Let $\mathcal{A}_n$ be the set of all words $(a_1, \cdots, a_n)$ of length $n$ such that $\sum_{i=1}^n a_i \leq (\alpha + \varepsilon)n$. Let 
$$
A_n:=\{x\in(0,1): \sum_{i=jn}^{jn+n-1} a_i(x) \leq (\alpha + \varepsilon)n \mbox{ for any } j\in\mathbb{N}\cup\{0\}\}
$$
be the set of points $x\in [0,1]$ such that $(a_1(x), a_2(x), \cdots)$ is a concatenation of words from $\mathcal{A}_n$. We have $A_n \subset F(\alpha + \varepsilon)$, thus $\dim_H A_n \leq \dim_H F(\alpha + \varepsilon) =: s$. $A_n$ is a limit set of an iterated function system generated by maps $f_{a_1,\ldots, a_n}$ with $(a_1,\cdots, a_n) \in \mathcal{A}_n$. This system satisfies the open set condition. Hence we have 

$$
\sum_{a_1\cdots a_n \in \mathcal{A}_n} \min_{x\in A_n} |f_{a_1,\cdots,a_n}'(x)|^{\dim_H A_n} \leq 1.
$$

By the bounded distortion property,

$$
\sum_{(a_1\cdots a_n) \in \mathcal{A}_n} |I_n(a_1,\ldots,a_n)|^s \leq K^s.
$$

Then by the exponential contraction property, 

$$
\sum_{a_1\cdots a_n \in \mathcal{A}_n} |I_n(a_1,\ldots,a_n)|^{s+\varepsilon} \leq K^s \lambda^{n\varepsilon}.
$$

Observe now that if $x\in W(\alpha)$ then for infinitely many $\{n_\ell\}_{\ell\in\mathbb{N}}$ we have $x\in I_{n_\ell}(a_1\ldots a_{n_\ell})$ for some $(a_1, \cdots, a_{n_\ell}) \in \mathcal{A}_{n_\ell}$. Hence, for every $N>0$

$$
W(\alpha) \subset \bigcup_{i>N} \bigcup_{(a_1, \cdots, a_i) \in \mathcal{A}_i} I_i(a_1,\cdots,a_i).
$$
This gives us a family of covers of $W(\alpha)$, using which one easily checks that

$$
\dim_H W(\alpha) \leq s+\varepsilon,
$$
from which we get $\dim_H W(\alpha) \leq \dim_H F(\alpha)$.

\end{proof}

\begin{remark}
According to the second referee of the paper, the inequality in the theorem can be sharpened into an equality. Here is a sketch of proof of the inverse inequality  $\dim_H W(\alpha) \geq \dim_H F(\alpha)$. Note that 
$$
F(\alpha) \subset \bigcup_{i>N} \bigcup_{(a_1, \cdots, a_i) \in \mathcal{A}_i} I_i(a_1,\cdots,a_i).
$$

Let $\eta_i$ satisfy $\sum_{(a_1,\ldots,\vartheta_i)\in\mathcal{A}_i}|I_n(a_1,\ldots,a_i)|^{\eta_i}=1$ for fixed $i$ and $\eta=\liminf_{i\rightarrow\infty} \eta_i$, then $\dim_H F(\alpha)\leq \eta$. Let $f(x)=a_1(x)$ for $x\in(0,1)$, 
$$
u(\alpha)=\sup\{h(\mu)/\lambda(\mu): \mu\in\mathcal{M}_G, \int f\rm{d}\mu\leq \alpha+\epsilon, \lambda(\mu)<\infty\}
$$
and
$$
v(\alpha)=\sup\{h(\mu)/\lambda(\mu): \mu\in\mathcal{M}_G, \int f\rm{d}\mu= \alpha+\epsilon, \lambda(\mu)<\infty\}.
$$
We can show that $u(\alpha)$ and $v(\alpha)$ depends continuously on $\alpha$, and they are the same if the measures $\mu$ are restricted on ergodic ones. By taking Bernoulli measures  ($G^i(x)$ invariant) with weights $|I_i(a_1,\ldots,a_i)|$ and then normalize to get invariant measures we have 

$$\eta\leq u(\alpha)=v(\alpha).$$  

According to \cite[Theorem 3.1]{IJ1}, let $\epsilon\rightarrow 0$, we have $v(\alpha)\leq \dim_H W(\alpha)$, so $\dim_H W(\alpha) \geq \dim_H F(\alpha)$.
\end{remark}

Now we can show our Theorem \ref{Theorem5} and \ref{Theorem1}  by virtue of Theorem \ref{Theorem8}.

Proof of Theorem \ref{Theorem5}:
\begin{proof}
This is an instant corollary of Theorem \ref{Theorem8} since for any irregular sequences $\{\varphi(n)\}_{n=1}^\infty$ with $\liminf_{n\rightarrow\infty}\frac{\varphi(n)}{n}=D<\infty$, we have $dim_H E_{\varphi(n)}\leq dim_H F(D)<1$. 
\end{proof}

Proof of Theorem \ref{Theorem1}:
\begin{proof}
This is a combinatorial result of Lemma \ref{lemma1}, Theorem \ref{Theorem3} and \ref{Theorem5}. 
\end{proof}

\section{Some further discussions}\label{section1}
Comparing our Proposition \ref{proposition1} with \cite[Theorem 1.1]{LR1}, we can see that for $1/2\leq \beta<1$ and a sequence $\{\varphi(n)\}_{n=1}^\infty$ with $\limsup_{n\rightarrow\infty}\frac{\log\log\varphi(n)}{\log n}=\beta$, the dimension of the set $dim_H E_{\varphi(n)}$ can still vary in a large scope (at least between $1/2$ and $1$). In this case, in order to get a full dimensional set (or say, an $s$-dimensional set, $0\leq s\leq 1$) $E_{\varphi(n)}$, not only order of $\varphi(n)$ matters, but also many other aspects. For example, in \cite[Theorem 1.2]{LR1}, differences of the two neighbouring terms $\varphi(n+1)-\varphi(n)$ (the term $\psi'(x)$ there in fact sets restrictions on $\varphi(n+1)-\varphi(n)$) play an importance role. Also, density of some ``special'' terms matters sometimes, as one can see from our examples in Proposition \ref{proposition1}, there $n^{1/N}$ describes density of the critical terms $\varphi(n_k)$. However, general conditions on these aspects for getting a full dimensional set are difficult for us to describe.

Now we would like to talk more about one of the techniques used in the proof of Theorem \ref{Theorem3}, that is, using bounds on the dimension of sets regarding individual quotients to bound dimension of sets regarding sums of the quotients. We can use the idea to do more in fact.

Since the dimensional theory of sets in continued fractions are developed, people have considered distributions of various terms, say, $a_n(x), s_n(x), T_n(x):=\max\{a_k(x): 1\leq k\leq n\}$ (see for example, \cite{Oka}, \cite{WX2}, \cite{LR1}, \cite{Ma}, \cite{FS}), $q_n(x)$ (see \cite{Bes} \cite{Goo}), etc., or hybrids of them (for example, see our Proposition \ref{proposition1}). Lots of results have been obtained since then. Although many results are closely related with each other, it seems that there are no systematic study of the links between these dimentional results on various terms. To the author, every dimentional result on one kind of the terms gives birth to results on other kinds of terms. For example, in 2008 Wu and Xu \cite[Theorem 1.1]{WX2} showed that for any $\alpha\geq 0$, $E(\alpha):=\{x\in(0,1): \lim_{n\rightarrow\infty} \frac{T_n(x)\log\log n}{n}=\alpha\}$, we have
\begin{center}
$dim_H E(\alpha)=1$.
\end{center}   
Now if we set 
\begin{center}
$V(\alpha)=\{x\in(0,1): a_n(x)\geq \frac{n}{2\log\log n}\alpha \ i. o.\ n\}$ 
\end{center} 
for any $\alpha\geq 0$, then we can see that
\begin{center}
$E(\alpha)\subset V(\alpha)$. 
\end{center}
So we get $dim_H V(\alpha)=1$ for any $\alpha\geq 0$. Of course we can also deal with $dim_H V(\alpha)$ by Theorem \ref{Theorem6} or a lemma of Good \cite[Lemma 6]{Goo}. There is another example on this idea. For $b, c>1$, let 
\begin{center}
$E(b,c)=\{x\in(0,1): a_n(x)\geq c^{b^n} \ i. o.\  n\}$.
\end{center} 
Then according to T. Luczak \cite[Theorem]{Luc} (see also \cite{FWLT}) or Theorem \ref{Theorem6},
\begin{center}
$dim_H E(b,c)=\frac{1}{b+1}$.
\end{center}
By this result we can show that
\begin{corollary}\label{corollary1}
For $b,c>1$ and $\varphi(n)=\sum_{k=1}^n c^{b^k}$, we have
\begin{center}
$dim_H E_{\varphi(n)}= \frac{1}{b+1}$. 
\end{center}
\end{corollary}
\begin{proof}
First, by Lemma \ref{lemma3} or \ref{lemma4}, we have
\begin{center}
$dim_H \{x\in(0,1): \frac{1}{2}c^{b^n}<a_n(x)<c^{b^n}, n\in\mathbb{N}\}=\frac{1}{b+1}$
\end{center}
for any $b,c>1$. Choose $\epsilon>0$ small enough such that $c-\epsilon>1$. For any $N\in\mathbb{N}$, let 
\begin{center}
$F_N(b,c-\epsilon)=\{x\in E_{\varphi(n)}: \frac{1}{2}{(c-\epsilon)}^{b^n}<a_n(x)<{(c-\epsilon)}^{b^n} \mbox{ for all } n\geq N\}$.
\end{center}
Then $dim_H F_N(b,c-\epsilon)\leq \frac{1}{b+1}$ for any $N\in\mathbb{N}$. We can show that
\begin{center}
$E_{\varphi(n)}\subset E(b, c-\epsilon)\cup\big(\cup_{N=1}^\infty F_N(b,c-\epsilon)\big)$.
\end{center}
Since $dim_H E(b, c-\epsilon)=\frac{1}{b+1}, dim_H F_N(b,c-\epsilon)\leq \frac{1}{b+1}$ for any $N\in\mathbb{N}$, then $dim_H E_{\varphi(n)}\leq \frac{1}{b+1}$. Moreover, as $\{x\in(0,1): c^{b^n}<a_n(x)<(1+\frac{1}{n})c^{b^n} \mbox{ for all } n\in\mathbb{N}\}\subset E_{\varphi(n)}$, by Lemma \ref{lemma4}, 
\begin{center}
$dim_H E_{\varphi(n)}\geq dim_H \{x\in(0,1): c^{b^n}<a_n(x)<(1+\frac{1}{n})c^{b^n} \mbox{ for all } n\in\mathbb{N}\}= \frac{1}{b+1}$,
\end{center}
so $dim_H E_{\varphi(n)}= \frac{1}{b+1}$.

\end{proof}


\end{document}